\documentclass[a4paper]{amsart}
\numberwithin{equation}{section}
\newtheorem{theorem}{Theorem}[section]
\newtheorem{corollary}{Corollary}[section]

\newtheorem{lemma}{Lemma}[section]

\theoremstyle{remark}

\usepackage{hyperref}
\usepackage{color}
\usepackage{graphicx}
\title[logarithmic coefficients]
 {On logarithmic coefficients of certain starlike functions related to the vertical strip\\
 The Journal Of Analysis, DOI: 10.1007/s41478-018-0157-7}
\begin{document}
\begin{abstract}
In the present paper two certain subclasses of the starlike functions associated with the vertical strip are considered.
 The main aim of this paper is to investigate some basic properties of these classes such as, subordination relations, sharp inequalities for sums involving logarithmic coefficients and estimate of logarithmic coefficients for functions belonging to these subclasses.
\end{abstract}

\author[R. Kargar]
       {Rahim Kargar}

\address{Young Researchers and Elite Club,
Ardabil Branch, Islamic Azad University, Ardabil, Iran}
       \email {rkargar1983@gmail.com}
\maketitle
\section{Introduction}
Throughout this paper $\Delta$ is the open unit disc on the complex plane $\mathbb{C}$. Let $\mathcal{H}$ be the class of all analytic functions in $\Delta$ and $\mathcal{A}$ be a subclass of $\mathcal{H}$ with the normalization $f(0)=f'(0)-1=0$. The subclass of $\mathcal{A}$ consisting of all univalent functions $f$ in $\Delta$ is denoted by $\mathcal{S}$. Let $\mathcal{S}^*$ and $\mathcal{K}$ denote the subclasses of $\mathcal{S}$ consisting of the normalized starlike and convex functions in $\Delta$, respectively. Also, we say that a function $f\in \mathcal{A}$ is close--to--convex, if there is a convex function $g$ such that
\begin{equation*}\label{CL}
  {\rm Re}\left\{\frac{f'(z)}{g'(z)}\right\}>0\quad(z\in\Delta).
\end{equation*}
Let $\mathcal{U}(\lambda)$ denote the set of all $f\in \mathcal{A}$ in $\Delta$ satisfying the condition
\begin{equation*}
  \left|\left(\frac{z}{f(z)}\right)^2 f'(z)-1\right|<\lambda\quad(z\in \Delta),
\end{equation*}
where $0<\lambda\leq 1$. For more details and interesting properties of the family $\mathcal{U}(\lambda)$, the reader may refer to \cite{obr BMalaysian}.
Also, let $\mathcal{G}(a)$ denote the class of locally univalent normalized analytic functions $f$ in $\Delta$ satisfying the condition
\begin{equation*}
  {\rm Re}\left\{1+\frac{zf''(z)}{f'(z)}\right\}<1+\frac{a}{2}\quad(a>0,\, z\in \Delta).
\end{equation*}
The class $\mathcal{G}(a)$ has been studied extensively by Kargar {\it et al.} \cite{kargarJMAA}, Maharana {\it et al.} \cite{Mah}, Obradovi\'{c} {\it et al.} \cite{ObSib}, and Ponnusamy and Sahoo \cite{PoSah}.

It is well--known that the logarithmic coefficients have had great impact in the development of the theory of univalent functions. For example, de Branges by use of this concept, was able to prove the famous Bieberbach's conjecture \cite{de Branges}.
The logarithmic coefficients $\gamma_n:=\gamma_n(f)$ of $f\in \mathcal{A}$ are defined by
\begin{equation}\label{log coef}
  \log\left\{\frac{f(z)}{z}\right\}=\sum_{n=1}^{\infty}2\gamma_n z^n\quad (z\in \Delta).
\end{equation}
As an example consider the rotation of Koebe function
\begin{equation*}\label{Koebe}
  k_\varepsilon(z)=\frac{z}{(1-\varepsilon z)^2}\quad (|\varepsilon|=1).
\end{equation*}
Then a simple calculation gives that
\begin{equation*}
\gamma_n(k_\varepsilon)=\frac{\varepsilon^n}{n}  \quad(n\geq1).
\end{equation*}
The inequality $|\gamma_n(f)|\leq 1/n$ holds for each starlike function $f\in\mathcal{S}$ and the equality is attained for the rotation of Koebe function, but it is false for the full class $\mathcal{S}$, even in order of magnitude.
Let $f\in \mathcal{S}$ and $f(z)=z+\sum_{n=2}^{\infty}a_n z^n$. Then by \eqref{log coef}, it follows that
\begin{equation*}
  \gamma_1=\frac{a_2}{2}\quad{\rm and} \quad \gamma_2=\frac{1}{2}\left(a_3-\frac{a_2^2}{2}\right)
\end{equation*}
and thus
the following sharp estimates hold
\begin{equation*}
  |\gamma_1|\leq1 \quad{\rm and}\quad |\gamma_2|\leq \frac{1}{2}(1+2e^{-2})\approx 0.635.
\end{equation*}
However, the sharp estimate of $|\gamma_n|$ when $n\geq 3$ and $f\in\mathcal{S}$ it is still open. For more explanation of this issue, it is necessary to point out that
there is a bounded and univalent function with logarithmic coefficients $\gamma_n$ such that $\gamma_n\neq O(n^{-0.83})$ \cite[p. 242]{Dur}. Also, there exists a close--to--convex function $f$ such that $|\gamma_n(f)|> 1/n$, \cite{girela}. In completing this entry Ye showed that the logarithmic coefficients $\gamma_n$ of each close--to--convex function $f$ in $\mathcal{S}$ satisfy $|\gamma_n(f)|\leq (A\log n)/n$, where $A$ is an absolute constant, see \cite{ye}.

Sharp inequalities are known for sums involving logarithmic coefficients. For instance, the logarithmic coefficients $\gamma_n$ of every function $f\in \mathcal{S}$ satisfy the sharp inequality
\begin{equation}\label{ineq. pi26}
\sum_{n=1}^{\infty}|\gamma_n|^2\leq \frac{\pi^2}{6}
\end{equation}
and the equality is attained for the Koebe function (see \cite[Theorem 4]{DurLeu}). Also, for each $f\in \mathcal{S}$ the sharp inequality
\begin{equation*}
\sum_{n=1}^{\infty}\left(\frac{n}{n+1}\right)^2|\gamma_n|^2\leq 4\sum_{n=1}^{\infty}\left(\frac{n}{n+1}\right)^2\frac{1}{n^2}=\frac{2\pi^2-12}{3}
\end{equation*}
holds (see \cite{Roth}).
Recently, Obradovi\'{c} {\it et al.} \cite{obradovicMM} proved that the logarithmic coefficients $\gamma_n$ of any $f\in \mathcal{U}(\lambda)$
satisfy the sharp inequality
\begin{equation*}
  \sum_{n=1}^{\infty}|\gamma_n|^2\leq \frac{1}{4}\left(\frac{\pi^2}{6}+2Li_2(\lambda)+Li_2(\lambda^2)\right),
\end{equation*}
where $Li_2$ denotes the dilogarithm function and that the logarithmic coefficients $\gamma_n$ of $f\in\mathcal{U}(1)$ satisfy the inequality \eqref{ineq. pi26}. Also, they proved that (see \cite[Theorem 2]{obradovicMM}) the logarithmic coefficients $\gamma_n$ of $f\in \mathcal{G}(a)$
satisfy the inequalities
\begin{equation*}
  \sum_{n=1}^{\infty}n^2|\gamma_n|^2\leq \frac{a}{4(a+2)}\quad(0<a\leq 1),
\end{equation*}
\begin{equation*}
  \sum_{n=1}^{\infty}|\gamma_n|^2\leq \frac{a^2}{4}Li_2\left((1+a)^{-2}\right)\quad(0<a\leq 1)
\end{equation*}
and
\begin{equation}\label{gamma n G}
  |\gamma_n|\leq \frac{a}{2(a+1)n}\quad(0<a\leq 1, n=1,2,\ldots).
\end{equation}
It's worth mentioning, that the above inequality \eqref{gamma n G} is not sharp. Very recently the sharp estimates for the initial logarithmic coefficients $\gamma_n$ of $f\in \mathcal{G}(a)$ where $0<a\leq 1$ and $n=1,2,3$ were obtained by Ponnusamy {\it et al.} (see \cite[Theorem 2.10]{PoShWi1}). They, also studied the logarithmic inverse coefficients, denoted by $\Gamma_n(F)$, of $f\in \mathcal{G}(a)$, where $F$ is the inverse function of $f\in \mathcal{G}(a)$. For more details see \cite{PoShWi0}.

In the sequel, we recall two certain subclasses of the starlike functions. Let $\mathcal{S}(\alpha,\beta)$ denote the class of all functions $f\in\mathcal{A}$ which satisfy the following two--sided inequality
\begin{equation*}\label{S(alpha,beta)}
  \alpha<{\rm Re}\left\{\frac{zf'(z)}{f(z)}\right\}<\beta\quad(\alpha<1,\, \beta>1).
\end{equation*}
The class $\mathcal{S}(\alpha,\beta)$ was introduced in \cite{KO2011} and studied in \cite{Kwon} and \cite{sim2013}. Also, we expanded the class $\mathcal{S}(\alpha,\beta)$ in \cite{Kargar(IJNAA)}.
By definition of subordination, $f\in \mathcal{S}(\alpha,\beta)$ if, and only if
\begin{equation}\label{f in S(al,be),iff}
  \frac{zf'(z)}{f(z)}\prec P_{\alpha,\beta}(z)\quad (z\in\Delta),
\end{equation}
where
\begin{equation}\label{P alpha beta}
  P_{\alpha,\beta}(z):=1+\frac{\beta-\alpha}{\pi}i \log \left(\frac{1-e^{2\pi
   i\frac{1-\alpha}{\beta-\alpha}}z}{1-z}\right).
\end{equation}
The function $P_{\alpha,\beta}(z)$ is convex univalent in $\Delta$ and has the form
\begin{equation}\label{P=1+}
  P_{\alpha,\beta}(z)=1+\sum_{n=1}^{\infty} B_n z^n,
\end{equation}
where
\begin{equation}\label{B-n}
B_n=\frac{\beta-\alpha}{n\pi}i \left(1-e^{2n\pi
i\frac{1-\alpha}{\beta-\alpha}}\right)\quad (n=1,2,\ldots)
\end{equation}
and maps
$\Delta$ onto a convex domain
\begin{equation*}\label{Omega}
    \Omega_{\alpha,\beta}:=\{ w\in \mathbb{C}: \alpha<
    {\rm Re}\, w<\beta\}
\end{equation*}
conformally. Recently, the function $P_{\alpha,\beta}(z)$ has been studied by many works, see for example \cite{KES(Siberian), Kargar(IJNAA), KO2011, Kwon, simBM, sim2013}.

Also, we say that a function $f\in\mathcal{A}$ belongs
to the class $\mathcal{M}(\delta)$, if $f$ satisfies
\begin{equation*}\label{1definition}
    1+\frac{\delta-\pi}{2 \sin \delta}<
    {\rm Re}\left\{\frac{zf'(z)}{f(z)}\right\} <
    1+\frac{\delta}{2\sin \delta} \quad (z\in\Delta),
\end{equation*}
where $\pi/2\leq \delta<\pi$. The class $\mathcal{M}(\delta)$ was introduced by Kargar {\it et al.} \cite{KES(Complex)}.
Moreover, by definition of subordination, $f\in\mathcal{M}(\delta)$ if, and
only if
\begin{equation*}\label{1l11}
\left(\frac{z f'(z)}{f(z)}-1\right)\prec
\mathcal{B}_\delta(z)
\quad (z\in\Delta),
\end{equation*}
where
\begin{equation}\label{B}
    \mathcal{B}_\delta(z):=\frac{1}{2i\sin
    \delta}\log\left(\frac{1+ze^{i\delta}}{1+ze^{-i\delta}}\right)\quad
    (z\in \Delta).
\end{equation}
The function $\mathcal{B}_\delta(z)$ due to Dorff \cite{MD} and studied in \cite{Dorff1997}, \cite{Dorff2012}, \cite{Dorff2015} and \cite{Dorff Rocki}.
The function $\mathcal{B}_\delta(z)$ is convex univalent in $\Delta$ and has the form
\begin{equation}\label{Bs}
    \mathcal{B}_\delta(z)=\sum_{n=1}^{\infty}A_n z^n \quad (z\in
    \Delta),
\end{equation}
where
\begin{equation}\label{A}
    A_n=\frac{(-1)^{(n-1)}\sin n\delta}{n \sin \delta}
    \quad (n=1,2,\ldots).
\end{equation}
The following lemma due to Ruscheweyh and Stankiewicz, will be useful in this paper.
\begin{lemma}\label{lem f ast g}{\rm (}see \cite{RUST}{\rm )}
Let $\phi,\varphi\in \mathcal H$ be any convex univalent functions in
$\Delta$. If $f(z)\prec \phi(z)$ and $g(z)\prec \varphi(z)$, then
\begin{equation*}\label{f*g}
    f(z)*g(z)\prec \phi(z)*\varphi(z) \quad (z\in \Delta),
\end{equation*}
where "*" denotes the Hadamard product.
\end{lemma}

In this paper, some
subordination relations among the classes $\mathcal{S}(\alpha,\beta)$ and $\mathcal{M}(\delta)$ are presented. These relations are then used to obtain
sharp estimates for sums involving their logarithmic coefficients. Also, the estimate of logarithmic coefficients for functions belonging to these subclasses are determined.

\section{Main Results}
One of the aims of this paper is the following theorem which will be useful in order to estimate of sums involving logarithmic coefficients of functions in the class $\mathcal{S}(\alpha,\beta)$.

\begin{theorem}\label{t2.1}
  Let $f(z)\in\mathcal A$, $\alpha<1$ and $\beta>1$. Also let $P_{\alpha,\beta}(z)$ be defined by \eqref{P alpha beta}. If $f(z)\in \mathcal{S}(\alpha,\beta)$, then
  \begin{equation}\label{log pS}
    \log\left\{\frac{f(z)}{z}\right\}\prec  \widehat{P}_{\alpha,\beta}(z),
  \end{equation}
  where
  \begin{equation}\label{hat P}
    \widehat{P}_{\alpha,\beta}(z):=\int_{0}^{z}\frac{P_{\alpha,\beta}(t)-1}{t}{\rm d}t,
  \end{equation}
  and $\widehat{P}_{\alpha,\beta}$ is convex univalent.
\end{theorem}
\begin{proof}
Let $f(z)\in\mathcal A$. If we define $p(z):=f(z)/z$, then $p(z)$ is analytic in $\Delta$ and $p(0)=1$. Also, since $f(z)\in \mathcal{S}(\alpha,\beta)$, therefore by \eqref{f in S(al,be),iff}, we have
  \begin{equation}\label{e1pt2.1}
    \frac{zp'(z)}{p(z)}=\frac{zf'(z)}{f(z)}-1\prec P_{\alpha,\beta}(z)-1\quad (z\in \Delta),
  \end{equation}
  where $P_{\alpha,\beta}$ is of the form \eqref{P alpha beta}. On the other hand,
   it is well--known that (see \cite{RU2}) the function
  \begin{equation*}\label{hwid}
  \widehat{h}(z)=\sum_{n=1}^{\infty}\frac{z^n}{n}
  \end{equation*}
is convex univalent in $\Delta$ and
\begin{equation*}\label{psi*h}
\psi(z)\ast \widehat{h}(z)=\int_{0}^{z}\frac{\psi(t)}{t}{\rm d}t\quad (\psi\in \mathcal{H}).
\end{equation*}
Now by Lemma \ref{lem f ast g} and from \eqref{e1pt2.1} we get
  \begin{equation}\label{e2pt2.1}
    \frac{zp'(z)}{p(z)}\ast\widehat{h}(z)\prec (P_{\alpha,\beta}(z)-1)\ast\widehat{h}(z)\quad (z\in \Delta).
  \end{equation}
Moreover by \eqref{e2pt2.1}, we can obtain \eqref{log pS}. On the other hand, since $P_{\alpha,\beta}(z)$ and $\widehat{h}(z)$ are convex univalent functions, by the P\`{o}lya--Schoenberg conjecture (this conjecture states that the class of convex
univalent functions is preserved under the convolution) that is proved by Ruscheweyh and Sheil--Small (see \cite{RUSS}), the function $\widehat{P}_{\alpha,\beta}(z)$ is convex univalent, too.
\end{proof}
Because $\widehat{P}_{\alpha,\beta}(z)$ is convex univalent, thus we get.
\begin{corollary}\label{c21}
  Let $f(z)\in \mathcal{S}(\alpha,\beta)$. Then
  \begin{equation*}
    \frac{f(z)}{z}\prec \exp \widehat{P}_{\alpha,\beta}(z)\quad (z\in \Delta),
  \end{equation*}
  where $\widehat{P}_{\alpha,\beta}(z)$ is given by \eqref{hat P}.
\end{corollary}
\begin{theorem}\label{th. cef log S al,be}
  For $\alpha<1$ and $\beta>1$, the logarithmic coefficients of $f\in \mathcal{S}(\alpha,\beta)$ satisfy the following inequality
  \begin{equation}\label{ineq S(al,be)}
    \sum_{n=1}^{\infty}|\gamma_n|^2\leq \frac{(\beta-\alpha)^2}{4\pi^2}\left(\frac{\pi^4}{45}-Li_4\left(e^{-2\pi i\frac{1-\alpha}{\beta-\alpha}}\right)-Li_4\left(e^{2\pi i\frac{1-\alpha}{\beta-\alpha}}\right)\right),
  \end{equation}
  where $Li_4$ is defined as following
  \begin{equation}\label{LI4}
    Li_4(z)=\sum_{n=1}^{\infty}\frac{z^n}{n^4}=-\frac{1}{2}\int_{0}^{1}\frac{\log^2(1/t)\log(1-tz)}{t}{\rm d}t.
  \end{equation}
  The result is sharp.
\end{theorem}
\begin{proof}
  Let $f\in \mathcal{S}(\alpha,\beta)$. Then by Theorem \ref{t2.1}, we have
  \begin{equation}\label{1-proof}
             \log\left\{\frac{f(z)}{z}\right\}\prec \widehat{P}_{\alpha,\beta}(z),
  \end{equation}
  where $\widehat{P}_{\alpha,\beta}(z)$ is defined in \eqref{hat P}. By using \eqref{P=1+} and \eqref{B-n}, one can rewrite $\widehat{P}_{\alpha,\beta}(z)$ as the following
  \begin{equation}\label{P2}
\widehat{P}_{\alpha,\beta}(z)=\sum_{n=1}^{\infty} \frac{\beta-\alpha}{\pi n^2}i \left(1-e^{2\pi n
i\frac{1-\alpha}{\beta-\alpha}}\right) z^n.
  \end{equation}
With placement of \eqref{log coef} and \eqref{P2} into \eqref{1-proof}, we get
\begin{equation*}\label{after with place}
  \sum_{n=1}^{\infty} 2\gamma_n z^n\prec \sum_{n=1}^{\infty} \frac{\beta-\alpha}{\pi n^2}i \left(1-e^{2\pi n
i\frac{1-\alpha}{\beta-\alpha}}\right) z^n.
\end{equation*}
Applying Rogosinski's theorem (see \cite{Rog} or \cite[Theorem 6.2]{Dur}), we obtain
\begin{align*}
  4\sum_{n=1}^{\infty} |\gamma_n|^2 &\leq \sum_{n=1}^{\infty}\frac{(\beta-\alpha)^2}{\pi^2 n^4}\left|i \left(1-e^{2\pi n
i\frac{1-\alpha}{\beta-\alpha}}\right)\right|^2 \\
  &= \frac{2(\beta-\alpha)^2}{\pi^2}\sum_{n=1}^{\infty}\frac{1}{n^4}\left(1-\cos 2\pi n \frac{1-\alpha}{\beta-\alpha}\right)\\
  &=\frac{2(\beta-\alpha)^2}{\pi^2}\left(\frac{\pi^4}{90}-\frac{1}{2}\left[Li_4\left(e^{-2\pi i\frac{1-\alpha}{\beta-\alpha}}\right)+Li_4\left(e^{2\pi i\frac{1-\alpha}{\beta-\alpha}}\right)\right]\right)
\end{align*}
and we get the inequality \eqref{ineq S(al,be)}. The inequality is sharp for the logarithmic coefficients of the function
  \begin{equation*}\label{sharp f2}
    \mathfrak{F}_{\alpha,\beta}(z)=z\exp \widehat{P}_{\alpha,\beta}(z),
  \end{equation*}
  where $\widehat{P}_{\alpha,\beta}(z)$ is given by \eqref{hat P}.
  A simple check gives us
  \begin{equation*}
    \gamma_n(\mathfrak{F}_{\alpha,\beta}(z))=\frac{\beta-\alpha}{2\pi n^2}i\left(1-e^{2\pi n i \frac{1-\alpha}{\beta-\alpha}}\right)
  \end{equation*}
  and concluding the proof.
\end{proof}
\begin{theorem}\label{th gamma n of S al be}
  Let $f\in \mathcal{A}$ belongs to the class $\mathcal{S}(\alpha,\beta)$ and $\gamma_n$ be the logarithmic coefficients of $f$. Then
    \begin{equation}\label{gamma n of S al be}
    |\gamma_n|\leq \frac{\beta-\alpha}{n\pi}\left|\sin \frac{\pi(1-\alpha)}{\beta-\alpha}\right|\quad (n\geq1, \alpha<1<\beta ).
  \end{equation}
\end{theorem}
\begin{proof}
  If $f\in \mathcal{A}$ belongs to the class $\mathcal{S}(\alpha,\beta)$, then by \eqref{f in S(al,be),iff} we have
\begin{equation*}
  z\frac{f'(z)}{f(z)}-1=z\left(\log \left\{\frac{f(z)}{z}\right\}\right)'\prec
  P_{\alpha,\beta}(z)-1,
\end{equation*}
where $P_{\alpha,\beta}$ is defined in \eqref{P alpha beta}. Moreover, in terms of the logarithmic coefficients
$\gamma_n$ of $f$ defined by \eqref{log coef} and \eqref{P=1+},
is equivalent to
\begin{equation*}
  \sum_{n=1}^{\infty}2n \gamma_n z^n\prec \sum_{n=1}^{\infty}B_n z^n.
\end{equation*}
Now by Rogosinski's theorem (see \cite[Theorem X]{Rog}), we get $2n|\gamma_n|\leq |B_1|$. Therefore the inequality \eqref{gamma n of S al be} follows. This completes the proof.
\end{proof}
It is clear that if $\beta\rightarrow +\infty$, then $\mathcal{S}(\alpha,\beta)\rightarrow \mathcal{S}^*(\alpha)$ (the class of starlike functions of order $\alpha$, where $0\leq \alpha<1$). Thus
we have the following result (see \cite[Remark 1]{obradovicMM}).
\begin{corollary}
If $f\in\mathcal{S}(\alpha,\beta)$ when $\beta\rightarrow +\infty$, then
\begin{equation*}
  |\gamma_n|\leq \frac{\beta-\alpha}{n\pi}\left|\sin \frac{\pi(1-\alpha)}{\beta-\alpha}\right|\leq \frac{\beta-\alpha}{n \pi}
  \times\frac{\pi(1-\alpha)}{\beta-\alpha}=\frac{1-\alpha}{n}\quad(n\geq1).
\end{equation*}
Indeed, if $f\in\mathcal{S}^*(\alpha)$ $(0\leq \alpha<1)$ and $\gamma_n$ is the corresponding logarithmic coefficients, then we have $|\gamma_n|\leq(1-\alpha)/n $ for $n\geq1$.
\end{corollary}
 Next, we have the following.
\begin{theorem}\label{th sub M delta}
  Let $\pi/2\leq \delta<\pi$. Also let $\mathcal{B}_\delta(z)$ and $A_n$ be defined by \eqref{B} and \eqref{A}, respectively. If $f(z)\in \mathcal{M}(\delta)$, then
  \begin{equation*}\label{log p}
    \log \left\{\frac{f(z)}{z}\right\}\prec \int_{0}^{z}\frac{\mathcal{B}_\delta(t)}{t}{\rm d}t.
  \end{equation*}
  Moreover,
  \begin{equation}\label{B2}
    \mathcal{\widetilde{B}}_\delta(z):=\int_{0}^{z}\frac{\mathcal{B}_\delta(t)}{t}{\rm d}t=\sum_{n=1}^{\infty}\frac{A_n}{n}z^n
  \end{equation}
  is a convex univalent function.
\end{theorem}
\begin{proof}
The proof is similar to the proof of the Theorem \ref{t2.1}, and thus we omit the details.
\end{proof}
Since $\mathcal{\widetilde{B}}_\delta(z)$ is a convex univalent function, thus we have.

\begin{corollary}
  If $f(z)\in \mathcal{M}(\delta)$, then
  \begin{equation*}
    \frac{f(z)}{z}\prec \exp \mathcal{\widetilde{B}}_\delta(z)\quad (z\in \Delta),
  \end{equation*}
  where $\mathcal{\widetilde{B}}_\delta(z)$ is of the form \eqref{B2}.
\end{corollary}

\begin{theorem}
  Let $f\in \mathcal{A}$ belongs to the class $\mathcal{M}(\delta)$ and $\pi/2\leq \delta<\pi$. Then the logarithmic coefficients of $f$ satisfy the inequality
  \begin{equation}\label{log ineq}
    \sum_{n=1}^{\infty}|\gamma_n|^2\leq\frac{1}{16 \sin^2 \delta}\left[\frac{\pi^4}{45}-
    Li_4\left(e^{-2i\delta }\right)-Li_4\left(e^{2i\delta }\right)\right],
  \end{equation}
  where $Li_4$ is defined in \eqref{LI4}. The result is sharp.
\end{theorem}
\begin{proof}
  Let $f\in \mathcal{M}(\delta)$. Then by Theorem \ref{th sub M delta}, we have
    \begin{align}\label{log f/z=A}
    \log\left\{\frac{f(z)}{z}\right\} &\prec \mathcal{\widetilde{B}}_\delta(z)\quad (z\in \Delta).
  \end{align}
  By using \eqref{log coef} and \eqref{Bs}, the relation \eqref{log f/z=A} implies that
  \begin{equation*}\label{sub gamma}
    \sum_{n=1}^{\infty}2\gamma_n z^n\prec \sum_{n=1}^{\infty}\frac{A_n}{n}z^n\quad (z\in \Delta).
  \end{equation*}
  Now by Rogosinski's theorem (see \cite{Rog} or \cite[Theorem 6.2]{Dur}), we get
  \begin{align*}
    4\sum_{n=1}^{\infty}|\gamma_n|^2 &\leq \sum_{n=1}^{\infty}\frac{1}{n^2}|A_n|^2\\
    &=\frac{1}{\sin^2 \delta}\sum_{n=1}^{\infty} \frac{\sin ^2 n\delta}{n^4}\\
    &=\frac{1}{\sin^2 \delta}\sum_{n=1}^{\infty}\left(\frac{1}{180}\left[\pi^4-45Li_4\left(e^{-2i\delta}\right)
    -45Li_4\left(e^{2i\delta}\right)\right]\right),
  \end{align*}
  where $Li_4$ is defined by \eqref{LI4}. Therefore the desired inequality \eqref{log ineq} follows. For the sharpness of \eqref{log ineq}, consider
  \begin{equation*}\label{sharp function}
    F_\delta(z)=z\exp \mathcal{\widetilde{B}}_\delta(z),
  \end{equation*}
  where $\mathcal{\widetilde{B}}_\delta(z)$ is defined by \eqref{B2}. It is easy to see that $F_\delta(z)\in \mathcal{M}(\delta)$ and $\gamma_n(F_\delta)=A_n/2n$, where $A_n$ is given by \eqref{A}. Therefore, we have the equality in \eqref{log ineq}. This is the end of proof.
\end{proof}
\begin{theorem}
  Let $\pi/2\leq \delta<\pi$. If $f\in\mathcal{A}$ belongs to the class $\mathcal{M}(\delta)$, then the logarithmic coefficients of $f$ satisfy
  \begin{equation*}
    |\gamma_n|\leq \frac{1}{2n}\quad (n\geq 1).
  \end{equation*}
\end{theorem}
\begin{proof}
  The proof is similar to the proof of the Theorem \ref{th gamma n of S al be}, and thus the details are omitted
\end{proof}
\noindent
{\bf Acknowledgements} This work is supported by Young Researchers and Elite Club, Ardabil branch. The author would like to thank the anonymous referee(s) for their careful readings, valuable suggestions and comments, which helped to improve the presentation of the paper.\\
\noindent
{\bf Compliance with ethical standards}\\
{\bf Research involving human participants and/or animals} This research does not contain any studies with human participants and/or animals performed by the author.\\
{\bf Conflict of interest} The author declares there is no conflict of interest related to this article.

\end{document}